%% file: main.tex
\documentclass{article}

\usepackage{mathtools}
\usepackage{amssymb}
\usepackage{amsthm}
\usepackage{mytheorems}
\usepackage{mymacros}

\theoremstyle{plain}
\newcommand{\thmA}{Theorem~A}
\newtheorem*{TheoremA}{\thmA}
\newcommand{\thmB}{Theorem~B}
\newtheorem*{TheoremB}{\thmB}
\newcommand{\thmC}{Theorem~C}
\newtheorem*{TheoremC}{\thmC}
\newcommand{\thmD}{Theorem~D}
\newtheorem*{TheoremD}{\thmD}
\newcommand{\thmE}{Theorem~E}
\newtheorem*{TheoremE}{\thmE}
\newcommand{\ocl}{Orbit Counting Lemma}
\newtheorem*{LemmaOCL}{\ocl}

\newcommand{\extF}{\widetilde{\openF}}

\begin{document}

\title{\bf On the finite subgroups of $2$-dimensional general linear groups}

\author{Paul Flavell}

\maketitle

\input{intro}

\input{plat}
\input{stab}

\input{np}

\input{pl}

\input{sf}

\input{exc}
\input{ptab}
\input{pte}

\input{zg}

\input{bibliography}


\end{document}

%% file: intro.tex
\section{Introduction} \label{intro}

Throughout this paper, $\openF$ is a field.
In a classic work,
Dickson\cite{D} classified the subgroups of the group $\psl{2}{\openF}$
in the case that $\openF$ is finite.
We will give a new proof of Dickson's classification as well as extending it slightly.
Pursuing the analysis further, 
we give a characterization of the finite subgroups of the groups $\pgl{2}{\openF}$
in terms of permutation groups. See \thmE.

The symbol $\pl{\openF}$ refers to the projective line over $\openF$
and the term \emph{platonic group} will be defined shortly
as will the symbol $\afiveBAR{\openF}$.

\begin{TheoremA}
    Let $\BAR{G}$ be a finite subgroup of $\pgl{2}{\openF}$.
    Let $\extF$ be an extension of $\openF$ such that every element
    of $\BAR{G}$ has a fixed point on $\pl{\extF}$.
    Set \[
        \Gamma = \bset{ \gamma \in \pl{\extF} }{ \mbox{%
            $\listset{ \gamma } = \fix{\pl{\extF}}{\BAR{g}}$ %
            for some $\BAR{g} \in \BAR{G}\nonid$
        }}.
    \]
    Then at least one of the following holds.
    \begin{itemize}
        \item[(a)]  $\BAR{G}$ has an orbit of length $1$ or $2$ on $\pl{\extF}$.

        \item[(b)]  The order of $\BAR{G}$ is not divisible by $\Char\openF$,
                    $\Gamma = \emptyset$ and
                    $\BAR{G}$ is a platonic group in its action on $\pl{\extF}$.

        \item[(c)]  The order of $\BAR{G}$ is divisible by $\Char\openF$,
                    $\Gamma$ is an orbit,
                    there is exactly one other nonregular orbit $\Delta$ on $\pl{\extF}$
                    and \[
                        \Gamma \BAR{h} = \pl{\openE}
                    \]
                    for some finite subfield $\openE \subseteq \openF$ and
                    $\BAR{h} \in \pgl{2}{\openF}$.
                    One of the following holds.
                    \begin{itemize}
                        \item[(i)]  $\BAR{G}^{\BAR{h}} = \psl{2}{\openE}$ or $\pgl{2}{\openE}$
                                    and $\card{\Delta} = \card{\openE}(\card{\openE} - 1)$.

                        \item[(ii)] $\Char\openF = 3$, $\openF$ contains a square root of $-1$,
                                    $\card{\openE} = 9$,
                                    $\BAR{G}^{\BAR{h}} = \afiveBAR{\openF} \isom \alt{5}$ and
                                    $\card{\Delta} = 12$.
                    \end{itemize}
    \end{itemize}
\end{TheoremA}
A field $\extF$ with the property required in the statement of \thmA\ always exists.
For example, the algebraic closure of $\openF$.
It is straightforward to describe the groups appearing in (a),
they are Frobenius groups, dihedral groups or subgroups thereof.
As an application of \thmA\ we prove the following.

\begin{TheoremB}[Dickson \cite{D}]
    Suppose that $\openF$ is finite with odd characteristic and generator $l$.
    Let \[
        G = \bblistgen{ \PMatrixC{1 & 0 \\ 1 & 1}, \PMatrixC{1 & l \\ 0 & 1} }.
    \]
    Then either
    \begin{itemize}
        \item[(a)]  $G = \spl{2}{\openF}$ or

        \item[(b)]  $\card{\openF} = 9, l^{2} = -1$ and
                    $G$ is isomorphic to the double cover of $\alt{5}$.
    \end{itemize}
\end{TheoremB}

Huppert\cite[8.27, p.213]{Hu} also gives a proof of Dickson's classification,
which is then used by Huppert and Blackburn\cite[7.6, p.494]{HBII} to
give a proof of \thmB.
Suzuki\cite[6.21, p.409]{S} determines the finite subgroups of $\spl{2}{\openF}$
when $\openF$ is algebraically closed and obtains Dickson's results as a corollary.
Gorenstein\cite[8.4, p.44]{Gor} gives a presentation of Dickson's proof of \thmB\
in modern terminology.
In \cite{F}, we give a direct proof of \thmB,
using different ideas to the ones employed here.

The existing proofs of Dickson's classification are largely combinatorial
involving counting element of subgroups and cosets.
The proof of \thmA\ presented here is also combinatorial but the focus
is on the action of a subgroup of $\pgl{2}{\openF}$ on the projective line.

Allowing $\openF$ to be any field and considering subgroups of $\pgl{2}{\openF}$
rather than just $\psl{2}{\openF}$ causes no difficulty --
except for an amusing diversion when $\openF$ is imperfect with characteristic $2$.
Next we describe the sequence of reasoning that leads to the proofs presented here.

The starting point is the following elementary result.

\begin{LemmaOCL}[Cauchy--Frobenius]
    Let the finite group $G$ act on the finite set $\Omega$ and
    let $n$ be the number of orbits.
    Then \[
        n = \frac{1}{\card{G}} \sum_{g \in G} \card{\fix{\Omega}{g}}.
    \]
\end{LemmaOCL}

\noindent A trivial corollary is that if
$G$ is transitive on $\Omega$ and $\card{\Omega} > 1$ then
some element of $G$ acts fixed point freely.
To describe a nontrivial application,
we first make a definition.

\begin{UnnumberedDefinition}
    Let the group $G$ act on the set $\Omega$.
    Then $G$ is a \emph{platonic group in its action on $\Omega$}
    if the action is faithful,
    there are exactly three nonregular orbits and
    the possibilities for the isomorphism type of $G$ and
    the nonregular orbit lengths are given below. \[
        \begin{array}{c|ccc}
            G & l_{1} & l_{2} & l_{3} \\ \hline
            \alt{4} & 6 & 4 & 4 \\
            \sym{4} & 12 & 8 & 6 \\
            \alt{5} & 20 & 20 & 12
        \end{array}
    \]
\end{UnnumberedDefinition}

\begin{TheoremC}[Well known]
    Let the finite group $G \not= 1$ act on the set $\Omega$.
    Suppose that \[
        \card{\fix{\Omega}{g}} = 2
    \] for all $g \in G\nonid$.
    Then one of the following holds.
    \begin{itemize}
        \item[(a)]  $\card{\fix{\Omega}{G}} = 2$.

        \item[(b)]  There is an orbit of length $2$.

        \item[(c)]  $G$ is a platonic group in its action on $\Omega$.
    \end{itemize}
\end{TheoremC}

\noindent This result is of course classical.
It is used to classify the finite groups of rotations of $\openR^{3}$.
See for example \cite[3.6]{Co} or \cite[chapter~6]{LW}.

The reader may be wondering what this has to do with the subgroups of $\pgl{2}{\openF}$.
Recall that the projective line over $\openF$,
which we denote by $\pl{\openF}$,
is the set of $1$-dimensional subspaces of $\openF^{2}$.
Then $\gl{2}{\openF}$ acts on $\pl{\openF}$.
The kernel of this action is $\zenter{\gl{2}{\openF}}$ so
$\pgl{2}{\openF}$ acts faithfully on $\pl{\openF}$.
This action is in fact sharply $3$-transitive.
In particular if $\BAR{g} \in \pgl{2}{\openF}\nonid$ then \[
    \card{\fix{\pl{\openF}}{\BAR{g}}} \in \listset{0,1,2}.
\]
It is elementary to show that
\begin{equation} \tag{1}
    \mbox{\em if $\BAR{g}$ has finite order and
        $\card{\fix{\pl{\openF}}{\BAR{g}}} = 1$ then
        $\BAR{g}$ has order $\Char\openF > 0$.
    }
\end{equation}
This suggests using the \ocl\ to study the finite subgroups of $\pgl{2}{\openF}$.
Fixed point free elements present an obstacle.
The solution is to extend the field.
Suppose that $\extF$ is an extension of $\openF$.
Of course $\gl{2}{\openF} \leq \gl{2}{\extF}$ and hence
we may regard $\pgl{2}{\openF}$ as a subgroup of $\pgl{2}{\extF}$.
Thus we have an action of $\pgl{2}{\openF}$ on $\pl{\extF}$.
If $\extF$ is algebraically closed then since every element of $\gl{2}{\openF}$
has an eigenvalue in $\extF$ we have
\begin{equation} \tag{2}
    \card{\fix{\pl{\openF}}{\BAR{g}}} \in \listset{1,2}
\end{equation}
for all $\BAR{g} \in \pgl{2}{\openF}\nonid$.
If $\openF$ is finite with order $q$ then we could
take $\extF$ to be finite of order $q^{2}$.

Observe that $(1), (2)$ and \thmC\ prove \thmA\ in the case that
$\Char\openF$ does not divide $\card{\BAR{G}}$.
For completeness therefore we include a proof of \thmC.
Moreover it also sets the scene for the following result,
which is key to proving \thmA\ in the case that
$\Char\openF$ divides $\card{\BAR{G}}$.

\newpage

\begin{TheoremD}
    Let the finite group $G$ act on the set $\Omega$.
    Assume the following.
    \begin{itemize}
        \item   $\card{\fix{\Omega}{g}} \in \listset{1,2}$ for all $g \in G\nonid$.

        \item   $\card{\fix{\Omega}{g}} = 1$ for some $g \in G\nonid$.

        \item   Each orbit has length at least $3$.
    \end{itemize}
    Set \[
        \Gamma = \bset{ \gamma \in \Omega }{\mbox{%
            $\listset{\gamma} = \fix{\Omega}{g}$ for some $g \in G\nonid$
        }}
    \]
    and for each $\gamma \in \Gamma$ let \[
        K_{\gamma} = \bset{ g \in G_{\gamma}\nonid }{%
            \fix{ \Omega \remove{\gamma} }{g} = \emptyset
        } \cup \blistset{1}.
    \]
    Fix $\gamma \in \Gamma$ and let $q = \card{K_{\gamma}}$.

    Then $\Gamma$ is an orbit,
    there is exactly one other nonregular orbit $\Delta$ and
    one of the following holds.
    \begin{itemize}
        \item[(a)]  $G$ is $2$-transitive on $\Gamma$,
                    $q \geq 3$,
                    $\card{\Gamma} = q + 1, \card{\Delta} = q(q-1)$
                    and if $\gamma' \in \Gamma \remove{\gamma}$
                    then $\card{G_{\gamma\gamma'}} = \half(q-1)$ or $q-1$.

        \item[(b)]  $G \isom \alt{5}, G_{\gamma} \isom \sym{3}$,
                    $K_{\gamma} \isom \cyclic{3}, q = 3, \card{\Gamma} = 10$ and
                    $\card{\Delta} = 12$.
    \end{itemize}
\end{TheoremD}

It follows from Frobenius' Theorem that the set $K_{\gamma}$
is a normal subgroup of $G_{\gamma}$.
However this fact is not needed in the proof of \thmD.

We mention that Cameron~\cite{C} and Iwahori~\cite{I} have also
proved variations of \thmC.
The author has been unable to establish of \thmD\ has been proved elsewhere.

In \S\ref{sf} we establish a characterization of the subfield subgroups
$\psl{2}{\openE}$ and $\pgl{2}{\openE}$.
This is used to show that case~(a) of \thmD\ gives case~(c)(i) of \thmA.
The argument used differs from the corresponding arguments of
Dickson, Huppert and Suzuki.

To complete the proof of \thmA\ we must deal with case~(b) of \thmD.
To this end we make the following definition.

\begin{UnnumberedDefinition}
    Whenever $\Char\openF = 3$ and
    $\openF$ contains a square root $i$ of $-1$,
    define \[
        \afive{\openF} = \bblistgen{ \PMatrixC{ 1 & 0 \\ 1 & 1}, \PMatrixC{ 1 & i \\ 0 & 1} }
    \]
    and $\afiveBAR{\openF}$ to be the image of $\afive{\openF}$ in $\pgl{2}{\openF}$.
\end{UnnumberedDefinition}

Using the isomorphism $\psl{2}{9} \isom \alt{6}$ it follows readily that
$\afive{\openF}$ is isomorphic to the double cover of $\alt{5}$.
See \S\ref{exc}.

Once \thmA\ has been proved,
\thmB\ follows quickly.

It is natural to ask if there are any further examples of groups that satisfy
the hypotheses of \thmD.
The answer is no.
Using a theorem of Zassenhaus we will establish the following,
which may be viewed as a characterization of the finite groups $\pgl{2}{\openF}$
and their subgroups in terms of permutation groups.
Again, the author has been unable to establish whether this result is already known.

\begin{TheoremE}
    Let the finite group $G$ act on the set $\Omega$.
    Assume that \[\mbox{
        $\card{\fix{\Omega}{g}} \in \listset{1,2}$ for all $g \in G\nonid$.
    }\]
    Then one of the following holds.
    \begin{itemize}
        \item[(a)]  There is an orbit of length one or two.

        \item[(b)]  $G$ is a platonic group in its action on $\Omega$.

        \item[(c)]  There exists a finite field $\openF$ such that
                    $G \isom \psl{2}{\openF}$ or $\pgl{2}{\openF}$.
                    More precisely
                    \begin{itemize}
                        \item[(i)]  there exists a nonregular orbit $\Gamma$ which may be identified
                                    with $\openF^{\infty}$ so that the elements of $G$ induce
                                    linear fractional transformations of $\Gamma$;

                        \item[(ii)] $G$ acts faithfully on $\Gamma$ and induces
                                    $\psl{2}{\openF}$ or $\pgl{2}{\openF}$ on $\Gamma$; and

                        \item[(iii)]there is exactly one other nonregular orbit and
                                    it has length \break \mbox{$\card{\openF}(\card{\openF} - 1)$}.
                    \end{itemize}

        \item[(d)]  $G \isom \alt{5}$,
                    there are exactly two nonregular orbits and
                    they have lengths $10$ and $12$.
    \end{itemize}
\end{TheoremE}

%% file: plat.tex
\section{The platonic case} \label{plat}

The proof of the following elementary result
is left to the reader.

\begin{Lemma} \label{plat.1}
    Let $G$ be a group is which every element has prime power order.
    If $\card{G} = 12, 24$ or $60$ then $G$ is
    isomorphic to $\alt{4}, \sym{4}$ or $\alt{5}$ respectively.
\end{Lemma}

\begin{proof}[Proof of \thmC]
    We have a finite group $G \not= 1$ acting on the set $\Omega$
    with the property
    $\card{ \fix{\Omega}{g} } = 2$ for all $g \in G\nonid$.
    Then there are nonregular orbits and, since $G$ is finite, only finitely many.
    Denote them by $\Omega_{1}, \ldots, \Omega_{n}$.
    We discard any regular orbits since they affect neither the hypothesis nor the conclusion.
    The \ocl\ gives $n\card{G} = \card{\Omega} + 2(\card{G} - 1)$ whence
    \begin{equation} \tag{1}
        (n - 2)\card{G} + 2 = \card{\Omega} = \card{\Omega_{1}} + \cdots \card{\Omega_{n}}.
    \end{equation}
    If $n = 1$ then $2 = \card{G} + \card{\Omega}$ so $\card{G} = 1$,
    contrary to hypothesis.
    If $n = 2$ then $\card{\Omega_{1}} = \card{\Omega_{2}} = 1$
    and Conclusion~(a) holds.
    Hence we assume that $n \geq 3$.

    For each $i$,
    let $s_{i}$ be the order of the stabilizer of an element of $\Omega_{i}$.
    Since $\Omega_{i}$ is nonregular we have
    $2 \leq s_{i} = \card{G}/\card{\Omega_{i}}$.
    Without loss, $s_{1} \leq \cdots \leq s_{n}$.

    Dividing $(1)$ by $\card{G}$ we obtain
    \begin{equation} \tag{2}
        n - 2 + \frac{2}{\card{G}} = \frac{1}{s_{1}} + \cdots \frac{1}{s_{n}}.
    \end{equation}
    Then $n - 2 < n/s_{1}$ so \[
        2 \leq s_{1} < \frac{n}{n-2} = 1 + \frac{2}{n-2}.
    \]
    This forces \[
        n = 3 \qtext{and} s_{1} = 2.
    \]
    Then $(2)$ becomes
    \begin{equation} \tag{3}
        \frac{1}{2} + \frac{2}{\card{G}} = \frac{1}{s_{2}} + \frac{1}{s_{3}}.
    \end{equation}
    In particular $1/2 < 2/s_{2}$ so $s_{2} \in \listset{2,3}$.
    If $s_{2} = 2$ then $2/\card{G} = 1/s_{3}$ so
    $\card{\Omega_{3}} = \card{G}/s_{3} = 2$ and Conclusion~(b) holds.
    Hence we assume that \[
        s_{2} = 3.
    \]
    Then $(3)$ becomes
    \begin{equation} \tag{4}
        \frac{1}{6} + \frac{2}{\card{G}} = \frac{1}{s_{3}}.
    \end{equation}
    In particular $1/6 < 1/s_{3}$ so $3 = s_{2} \leq s_{3} < 6$ and \[
        s_{3} \in \listset{3,4,5}.
    \]
    Since every element belongs to a stabilizer,
    we have shown that every element has prime power order.

    \medskip

    \noindent {\bf The case $\mathbf{s_{3} = 3}$.}
    Then $(4)$ forces $\card{G} = 12$.
    Hence $\card{\Omega_{1}} = 6$ and $\card{\Omega_{2}} = 4$
    and $\card{\Omega_{3}} = 4$.

    \medskip

    \noindent {\bf The case $\mathbf{s_{3} = 4}$.}
    Then $(4)$ forces $\card{G} = 24$.
    Hence $\card{\Omega_{1}} = 12, \card{\Omega_{2}} = 8$ and $\card{\Omega_{3}} = 6$.

    \medskip

    \noindent {\bf The case $\mathbf{s_{3} = 5}$.}
    Then $(4)$ forces $\card{G} = 60$.
    Hence $\card{\Omega_{1}} = 30, \card{\Omega_{2}} = 20$ and $\card{\Omega_{3}} = 12$.

    \medskip

    \noindent Applying Lemma~\ref{plat.1},
    in each case $G$ is a platonic group in its action on $\Omega$,
    completing the proof.
\end{proof}

Iwahori\cite{I} obtains more information when Conclusion~(b) holds.

%% file: stab.tex
\section{Stabilizers} \label{stab}

The following lemma will be used to
analyze the stabilizer of a point in the proof of \thmD.
The reader familiar with Frobenius groups will see,
that apart from trivial cases,
the group $G$ is an Frobenius group and
$K$ is its Frobenius kernel.
However, none of that theory is required.

\begin{Lemma} \label{stab.1}
    Suppose the finite group $G$ acts on the finite set $\Omega$
    and that $\card{\fix{\Omega}{g}} \in \listset{0,1}$
    for all $g \in G\nonid$.
    Let \[
        K = \bset{ g \in G }{ \fix{\Omega}{g} = \emptyset } \cup \blistset{1}.
    \]
    Then there is an orbit of length $\card{K}$ and
    any remaining orbits are regular.
\end{Lemma}
\begin{proof}
    If all the orbits are regular then
    they all have length $\card{G}$ and $K = G$
    so the conclusion holds in this case.
    Hence we discard any regular orbits and
    suppose that $\Omega$ consists of $n \geq 1$ nonregular orbits.
    The \ocl\ yields
    \begin{align*}
        n\card{G} &= \card{\Omega} + \card{G \setminus K} \\
    \intertext{whence}
        (n-1)\card{G} &= \card{\Omega} - \card{K}.
    \end{align*}
    The length of any nonregular orbit is a proper divisor of $\card{G}$
    whence $\card{\Omega} \leq (n/2)\card{G}$ and then \[
        (n-1)\card{G} < \frac{n}{2}\card{G}.
    \]
    This forces $n = 1$.
    Then $\card{\Omega} = \card{K}$ and the proof is complete.
\end{proof}

%% file: np.tex
\section{The nonplatonic case} \label{np}

We prove \thmD.
Thus we have a finite group $G$ acting on a set $\Omega$
satisfying the following.

\begin{itemize}
    \item   $\card{\fix{\Omega}{g}} \in \listset{1,2}$ for all $g \in G\nonid$.

    \item   $\card{\fix{\Omega}{g}} = 1$ for some $g \in G\nonid$.

    \item   Each orbit has length at least $3$.
\end{itemize}
Set \[
    \Gamma = \bset{ \gamma \in \Omega }{\mbox{%
        $\listset{\gamma} = \fix{\Omega}{g}$ for some $g \in G\nonid$
    }}
\]
and for each $\gamma \in \Gamma$ let \[
    K_{\gamma} = \bset{ g \in G_{\gamma}\nonid }{%
            \fix{ \Omega \remove{\gamma} }{g} = \emptyset
        } \cup \listset{1}.
\]

Notice that $\Gamma$ is a union of orbits,
each of which is nonregular.
By hypothesis $\Gamma \not= \emptyset$.
We discard any regular orbits since
they affect neither the conclusion nor the hypothesis.
Since $\card{\fix{\Omega}{g}} \in \listset{1,2}$ for all $g \in G\nonid$
it now follows that $\Omega$ is finite.

Set \[
    \Delta = \Omega \setminus \Gamma.
\]
Let $m$ and $n$ be the numbers of orbits contained
in $\Delta$ and $\Gamma$ respectively.
Then $n \geq 1$.
Since every element of $G$ has a fixed point on $\Omega$
it follows that $G$ is not transitive.
Hence $m + n \geq 2$.

For each $i \in \listset{1,2}$ set \[
    G_{i} = \bset{ g \in G\nonid }{ \card{\fix{\Omega}{g}} = i }
\]
so that
\begin{equation} \tag{1}
     G = 1 \dotcup G_{1} \dotcup G_{2}.
\end{equation}
The definitions of $G_{1}$ and $\Gamma$
imply that $\card{G_{1}} \geq \card{\Gamma}$.

The \ocl\ and (1) yield
\begin{align}
    (m+n)\card{G} &= \card{\Omega} + \card{G_{1}} + 2\card{G_{2}} \notag \\
                  &= \card{\Delta} + \card{\Gamma} + \card{G_{1}} + 2(\card{G} - \card{G_{1}} - 1) \notag \\
\intertext{so}
    (m+n-2)\card{G} &= \card{\Delta} - (\card{G_{1}} - \card{\Gamma}) - 2 \notag \\
                    &< \card{\Delta}. \tag{2}
\end{align}
Since $\Delta$ consists of $m$ nonregular orbits
we have $\card{\Delta} \leq (m/2)\card{G}$ and so \[
    \frac{m}{2} + n < 2.
\]
Now $n \geq 1$ and $m + n \geq 2$ so
this forces $n = m = 1$.
We have shown that \[\mbox{
    $\Gamma$ and $\Delta$ are both orbits.
}\]
Moreover from (2) we have
\begin{equation} \tag{3}
    \card{G_{1}} + 2 = \card{\Gamma} + \card{\Delta}.
\end{equation}

Fix $\gamma \in \Gamma$ and set $q = \card{K_{\gamma}}$.
We have $G_{1} = \dotcup_{\alpha \in \Gamma} \bigl( G_{1} \cap G_{\alpha} \bigr)$.
Now $\Gamma$ is an orbit so
$\card{G_{1} \cap G_{\alpha}} = \card{G_{1} \cap G_{\gamma}}$
for all $\alpha \in \Gamma$ whence
\begin{equation} \tag{4}
    \card{G_{1}} = \card{G_{1} \cap G_{\gamma}}\card{\Gamma}.
\end{equation}
Also $G_{1} \cap G_{\gamma} = K_{\gamma}\nonid$ so $(3)$ and $(4)$ yield
\begin{equation} \tag{5}
    (q-2)\card{\Gamma} + 2 = \card{\Delta}.
\end{equation}
By hypothesis,
each orbit of $G$ has length at least $3$ so $\card{\Delta} \geq 3$
and then $q \geq 3$.

Multiplying $(5)$ by $\card{G_{\gamma}}$ gives
$(q-2)\card{G} + 2\card{G_{\gamma}} = \card{\Delta}\card{G_{\gamma}}$.
Now $\Delta$ is an orbit so $\card{\Delta}$ divides $\card{G}$,
whence
\begin{equation} \tag{6}
    \card{\Delta} \qtext{divides} 2\card{G_{\gamma}}.
\end{equation}

Consider the action of $G_{\gamma}$ on $\Omega \remove{\gamma}$.
Now $\card{\fix{\Omega \remove{\gamma}}{g}} \in \listset{0,1}$
for all $g \in G_{\gamma}\nonid$ so Lemma~\ref{stab.1},
with $G_{\gamma}$ in the role of $G$ and
$\Omega \remove{\gamma}$ in the role of $\Omega$,
implies that $G_{\gamma}$ has one orbit of length $q$ and
any remaining orbits are regular.
Since $\Gamma \remove{\gamma}$ is
a union of orbits of $G_{\gamma}$ we have
\begin{equation} \tag{7}
    \card{\Gamma} = 1 + aq + b\card{G_{\gamma}}
\end{equation}
for some integers $a \in \listset{0,1}$ and $b \geq 0$.

From $(5), (6)$ and $(7)$ we have
\begin{equation} \tag{8}
    (q-2)(1 + aq + b\card{G_{\gamma}}) + 2 = \card{\Delta} %
    \qtext{divides} 2\card{G_{\gamma}}.
\end{equation}
Since $q \geq 3$ this forces $b \in \listset{0,1}$.
Recall that $\Gamma$ is an orbit of $G$ and
by hypothesis each orbit has length at least $3$.
Thus $(7)$ implies that $(a,b) \not= (0,0)$.
We have three cases.

\medskip

\noindent {\bf The case $\mathbf{(a,b) = (1,0)}$.}
Then $G_{\gamma}$ has one orbit on $\Gamma \remove{\gamma}$ and
it has length $q$.
Hence $G$ is $2$-transitive on $\Gamma$ and
$\card{\Gamma} = q + 1$.
Then $(5)$ implies that $\card{\Delta} = (q-2)(q+1) + 2 = q(q-1)$.
Let $\gamma' \in \Gamma \remove{\gamma}$ so that
$\card{G_{\gamma}} = \card{G_{\gamma\gamma'}}q$ because
$G_{\gamma}$ is transitive on $\Gamma \remove{\gamma}$.
Then $(6)$ implies that $q-1$ divides $2\card{G_{\gamma\gamma'}}$.
No nonidentity element of $G$ fixes $3$ points so
the action of $G_{\gamma\gamma'}$ on $\Gamma \remove{\gamma,\gamma'}$
is semiregular.
Hence $\card{G_{\gamma\gamma'}}$ divides $q-1$.
It follows that $\card{G_{\gamma\gamma'}} = \half(q-1)$ or $q-1$.
Conclusion~(a) holds.

\medskip

\noindent {\bf The case $\mathbf{(a,b) = (0,1)}$.}
Then $(8)$ yields \[
    (q-2)(1 + \card{G_{\gamma}}) + 2 \qtext{divides} 2\card{G_{\gamma}}.
\]
Hence $q = 3 = \card{G_{\gamma}}$.
Then in $(7)$ we may replace $(a,b)$ by $(1,0)$ and apply the previous case.

\medskip

\noindent {\bf The case $\mathbf{(a,b) = (1,1)}$.}
From $(8)$ we have $q = 3$ and
$1 + 3 + \card{G_{\gamma}} + 2 = \card{\Delta} = 2\card{G_{\gamma}}$ so
$\card{G_{\gamma}} = 6$ and $\card{\Delta} = 12$.
Now $G_{\gamma}$ has an orbit of length $q = 3$ and
since no nonidentity element of $G$ fixes $3$ points we have $G_{\gamma} \isom \sym{3}$.
From $(7)$, $\card{\Gamma} = 10$ so
$\card{G} = \card{G_{\gamma}}\card{\Gamma} = 60$.

Let $\delta \in \Delta$ so that $\card{G_{\delta}} = 5$.
Recall that $\Gamma$ and $\Delta$ are the orbits of $G$
and that each element of $G$ has at least one fixed point.
Then each nonidentity element of $G$ is contained in  a
conjugate of $G_{\gamma}$ or $G_{\delta}$ and
so has order $2,3$ or $5$.
Lemma~\ref{plat.1} implies that $G \isom \alt{5}$
so Conclusion~(b) holds and
the proof of \thmD\ is complete.

%% file: pl.tex
\section{The projective line} \label{pl}

We remind the reader of some basic properties
of the action of $\pgl{2}{\openF}$ on
the projective line $\pl{\openF}$.
Recall that $\pl{\openF}$ is
the set of $1$-dimensional subspaces of $\openF^{2}$.
For each $(x,y) \in \openF^{2} \remove{(0,0)}$
let $[x:y]$ be the subspace spanned by $(x,y)$.
Then $[x:y] = [x':y']$ if and only if $xy' = yx'$.
Consequently \[
    \pl{\openF} = \bset{  [x:1] }{ x \in \openF } \cup \blistset{ [1:0] }.
\]

The group $\gl{2}{\openF}$ acts on $\pl{\openF}$.
The kernel of this action is $\zenter{\gl{2}{\openF}}$.
Thus $\pgl{2}{\openF}$ acts faithfully on $\pl{\openF}$.
This action is sharply $3$-transitive.
Since $\spl{2}{\openF} \cap \zenter{\gl{2}{\openF}} = \zenter{\spl{2}{\openF}}$
we may abuse notation to suppose that $\psl{2}{\openF} \normal \pgl{2}{\openF}$.

We will have occasion to work simultaneously in $\gl{2}{\openF}$
and $\pgl{2}{\openF}$.
Bar notation $\BAR{H}$ will be used to refer to subgroups of $\pgl{2}{\openF}$.
Moreover, if $H \leq \gl{2}{\openF}$ we use $\BAR{H}$ to denote
the image of $H$ in $\pgl{2}{\openF}$.
Similarly for elements.

Next we introduce some subgroups that play a key role.
Let $\alpha \in \pl{\openF}$.
Now $\alpha$ is a subspace of $\openF^{2}$ so
$\gl{2}{\openF}_{\alpha}$, the stabilizer of $\alpha$,
acts on both $\alpha$ and the quotient space $\openF^{2} / \alpha$.
Define \[
    \unip{\alpha}{\openF} = \bset{ g \in \gl{2}{\openF} }{\mbox{%
        $g$ acts trivially on $\alpha$ and on $\openF^{2} / \alpha$ %
        }}
\]
and let \[\mbox{%
    $\unipBAR{\alpha}{\openF}$ be the image of $\unip{\alpha}{\openF}$ in $\pgl{2}{\openF}$. %
}\]
Then $\unip{\alpha}{\openF} \normal \gl{2}{\openF}_{\alpha}$ and
$\unipBAR{\alpha}{\openF} \normal \pgl{2}{\openF}_{\alpha}$.

\begin{Lemma} \label{pl.1}
    Let $\alpha \in \pl{\openF}$ and $g \in \gl{2}{\openF}_{\alpha}$.
    \begin{itemize}
        \item[(a)]  $\unip{\alpha}{\openF} \isom \unipBAR{\alpha}{\openF} \isom \openF\additive$.

        \item[(b)]  The actions of $\unip{\alpha}{\openF}$ and $\unipBAR{\alpha}{\openF}$
                    on $\pl{\openF} \remove{\alpha}$ are regular.

        \item[(c)]  \begin{itemize}
                        \item[(i)]  If $\BAR{g} \in \unipBAR{\alpha}{\openF}\nonid$ then
                                    $\BAR{g}$ is fixed point free on
                                    $\pl{\openF} \remove{\alpha}$ and
                                    $\listgen{ \BAR{g} }$ is isomorphic to
                                    a subgroup of $\openF\additive$.

                        \item[(ii)] If $\BAR{g} \not\in \unipBAR{\alpha}{\openF}$ then
                                    $\BAR{g}$ has a unique fixed point on
                                    $\pl{\openF} \remove{\alpha}$ and
                                    $\listgen{ \BAR{g} }$ is isomorphic to
                                    a subgroup of $\openF\mult$.
                    \end{itemize}

        \item[(d)]  If $\alpha = [1:0]$ then \[
                        \unip{\alpha}{\openF} = %
                        \bbset{ \PMatrixC{ 1 & 0 \\ c & 1 } }{ c \in \openF }.
                    \]
    \end{itemize}
\end{Lemma}
\begin{proof}
    Since $\gl{2}{\openF}$ is transitive on $\pl{\openF}$
    we may suppose that $\alpha = [1:0]$.
    Then \[
        g = \PMatrixC{ a & 0 \\ c & d}
    \]
    for some $a,d \in \openF \remove{0}$ and $c \in \openF$.
    Now $g$ acts as multiplication by $a$ on $\alpha$ and
    multiplication by $d$ on $\openF^{2} / \alpha$.
    This proves (d) and (a) follows.
    We have \[
        [0:1] \PMatrixC{1 & 0 \\ c & 1} = [c:1]
    \]
    so (b) follows because
    $\pl{\openF} = \set{ [x:1] }{ x \in \openF } \cup \listset{[1:0]}$.

    Note that (c)(i) follows from (b) and (a).
    It remains to prove (c)(ii),
    so suppose that $\BAR{g} \not\in \unipBAR{\alpha}{\openF}$.
    Then $a \not= d$.
    Moreover if $x \in \openF$ then \[
        [x:1]g = [x:1] \qtext{if and only if} (d-a)x = c.
    \]
    Thus $[(d-a)^{-1}c:1]$ is the unique  fixed point of $\BAR{g}$ on
    $\pl{\openF} \setminus \listset{\alpha}$.
    Finally, \[
        \pgl{2}{\openF}_{\alpha\beta} \isom \openF\mult
    \]
    whenever $\alpha, \beta \in \pl{\openF}$ are distinct
    and the proof is complete.
\end{proof}

The following two lemmas show that the situation considered
in \thmB\ is quite general.

\begin{Lemma} \label{pl.2}
    Let $p = \Char\openF$ and suppose $p > 0$.
    \begin{itemize}
        \item[(a)]  Let $g \in \gl{2}{\openF}\nonid$ be a $p$-element.
                    Then $g \in \unip{\alpha}{\openF}$ for some $\alpha \in \pl{\openF}$.
                    In particular $\card{\fix{\pl{\openF}}{g}} = 1$.

        \item[(b)]  Let $\BAR{g} \in \psl{2}{\openF}\nonid$.
                    Then \[\mbox{%
                        $\card{\fix{\pl{\openF}}{\BAR{g}}} = 1$
                        if and only if
                        $\BAR{g}$ is a $p$-element.
                    }\]

        \item[(c)]  Suppose there are $p$-elements contained in
                    $\pgl{2}{\openF} \setminus \psl{2}{\openF}$.
                    Then $p = 2$ and $\openF$ is imperfect.
    \end{itemize}
\end{Lemma}
\begin{proof}
    (a). Since $\Char\openF = p$,
    the only eigenvalue for $g$ is $1$.
    Thus $g \in \unip{\alpha}{\openF}$ where
    $\alpha$ is the $1$-eigenspace of $g$.
    Lemma~\ref{pl.1}(c) implies that
    $\listset{\alpha} = \fix{\pl{\openF}}{g}$.

    (b). Suppose that $\fix{\pl{\openF}}{\BAR{g}} = \listset{\alpha}$.
    Lemma~\ref{pl.1} implies $\BAR{g} \in \unipBAR{\alpha}{\openF} \isom \openF\additive$
    so $\BAR{g}$ is a $p$-element.
    Conversely suppose that $\BAR{g}$ is a $p$-element.
    Now $\BAR{g} \in \psl{2}{\openF}$ and
    $\card{\zenter{\spl{2}{\openF}}} = 1$ or $2$ so
    $\BAR{g}$ has an inverse image $g \in \psl{2}{\openF}$
    that is a $p$-element.
    Then (a) implies that
    $\card{\fix{\pl{\openF}}{\BAR{g}}} = 1$.

    (c). Because
    $\pgl{2}{\openF} / \psl{2}{\openF} \isom \openF\mult / ( \openF\mult )^{2}$.
\end{proof}

\begin{Lemma} \label{pl.3}
    Let $\alpha, \beta \in \pl{\openF}$ be distinct,
    $a \in \unip{\alpha}{\openF}\nonid$ and
    $b \in \unip{\beta}{\openF}$.
    Then there exists $g \in \gl{2}{\openF}$ and $l \in \openF$
    such that
    $\alpha g = [1:0], \beta g = [0:1]$, \[
        a^{g} = \PMatrixC{1 & 0 \\ 1 & 1} \qtext{and}
        b^{g} = \PMatrixC{1 & l \\ 0 & 1}.
    \]
\end{Lemma}
\begin{proof}
    Since $\unip{\alpha}{\openF}$ is regular on $\pl{\openF} \remove{\alpha}$
    it follows that $\alpha, \beta$ and $\beta a$ are distinct.
    By $3$-transitivity,
    $\alpha g = [1:0], \beta g = [0:1]$ and $\beta ag = [1:1]$
    for some $g \in \gl{2}{\openF}$.
    Now $a^{g} \in \unip{[1:0]}{\openF}$ and $b^{g} \in \unip{[0:1]}{\openF}$ so
    $a^{g} = \PSmallMatrixC{1 & 0 \\ m & 1}$ and
    $b^{g} = \PSmallMatrixC{1 & l \\ 0 & 1}$
    for some $m,l \in \openF$.
    We have \[
        [1:1] = \beta ag = \beta ga^{g} = [0:1]\PMatrixC{1 & 0 \\m & 1} = [m:1]
    \]
    so $m = 1$.
\end{proof}

%% file: sf.tex
\section{Subfield subgroups} \label{sf}

Suppose that $\openE$ is a subfield of $\openF$.
Clearly $\gl{2}{\openE} \leq \gl{2}{\openF}$.
Now $\zenter{\gl{2}{\openE}} = \gl{2}{\openE} \cap \zenter{\gl{2}{\openF}}$
so we may identify $\pgl{2}{\openE}$ with
the image of $\gl{2}{\openE}$ in $\pgl{2}{\openF}$
so that \[
    \pgl{2}{\openE} = \BAR{\gl{2}{\openE}} \leq \pgl{2}{\openF}.
\]
Similarly for $\psl{2}{\openE}$.

The map $[x:y]_{\openE} \mapsto [x:y]_{\openF}$
defines an injection $\pl{\openE} \longrightarrow \pl{\openF}$
that is compatible with the action of $\gl{2}{\openE}$.
Hence we identify $\pl{\openE}$ with its image in $\pl{\openF}$
so that
\begin{align*}
    \pl{\openE} &= \bset{ [x:1] }{ x \in \openE } \cup \blistset{ [1:0] } \\
                &\subseteq \bset{ [x:1] }{ x \in \openF } \cup \blistset{ [1:0] } %
                = \pl{\openF}.
\end{align*}
The notation $\pgl{2}{\openF}_{(\pl{\openE})}$ denotes
the setwise stabilizer of $\pl{\openE}$ in $\pgl{2}{\openF}$.

\begin{Lemma} \label{sf.1}
    Let $\openE$ be a subfield of $\openF$.
    \begin{itemize}
        \item[(a)]  If $\BAR{g} \in \pgl{2}{\openF}$ and
                    $\card{ \pl{\openE} \cap \pl{\openE}\BAR{g} } \geq 3$
                    then $\pl{\openE} = \pl{\openE}\BAR{g}$ and
                    $\BAR{g} \in \pgl{2}{\openE}$.

        \item[(b)]  $\pgl{2}{\openF}_{(\pl{\openE)}} = \pgl{2}{\openE}$.

        \item[(c)]  $\gl{2}{\openF}_{(\pl{\openE)}} = \zenter{\gl{2}{\openF}}\gl{2}{\openE}$.
    \end{itemize}
\end{Lemma}
\begin{proof}
    (a). Choose distinct $\alpha, \beta, \gamma \in \pl{\openE}$ with
    $\alpha\BAR{g}, \beta\BAR{g}, \gamma\BAR{g} \in \pl{\openE}$.
    Since $\pgl{2}{\openE}$ is $3$-transitive on $\pl{\openE}$
    we have $\alpha\BAR{g} = \alpha\BAR{h}, \beta\BAR{g} = \beta\BAR{h}$
    and $\gamma\BAR{g} = \gamma\BAR{h}$ for some $\BAR{h} \in \pgl{2}{\openE}$.
    The sharp $3$-transitivity of $\pgl{2}{\openF}$ forces $\BAR{g} = \BAR{h}$.
    In particular, $\BAR{g} \in \pgl{2}{\openE}$.

    (b) and (c) are immediate consequences of (a).
\end{proof}

The following result is an easy special case of
the Glauberman-Heimbeck characterization of finite fields
\cite[Lemma~4.3]{G} and \cite{He}.
The identity used in the proof is due to Hua.

\begin{Lemma} \label{sf.2}
    Let $\openE \subseteq \openF$.
    Assume the following.
    \begin{itemize}
        \item   $\openE$ is a subgroup of $\openF\additive$.

        \item   $a^{-1} \in \openE$ for all $a \in \openE \remove{0}$.

        \item   $1 \in \openE$.

        \item   If $\Char\openF = 2$ then $\openE$ is finite.
    \end{itemize}
    Then $\openE$ is a field.
\end{Lemma}
\begin{proof}
    We must show that $\openE$ is closed under multiplication.
    Let $a, b \in \openE$.
    We claim that
    \begin{equation} \tag{$*$}
        a^{2}b \in \openE.
    \end{equation}
    If $a = 0, b = 0$ or $a = -b^{-1}$ then this is visibly true.
    If not then \[
        \frac{1}{a} - \frac{1}{a+b^{-1}} = %
        \frac{b^{-1}}{a(a+b^{-1})} = \frac{1}{a(ab+1)}
    \]
    so \[
        a^{2}b = \left( \frac{1}{a} - \frac{1}{a + b^{-1}} \right)^{-1} - a \in \openE.
    \]

    Suppose that $\Char\openF \not= 2$.
    Then $1/2 \in \openE$ and then \[
        ab = (a + b)^{2}\half - a^{2}\half - b^{2}\half \in \openE.
    \]

    Suppose that $\Char\openF = 2$.
    The map $\sigma: \openF \longrightarrow \openF$ defined by
    $x\sigma = x^{2}$ is injective.
    By (1), with $b = 1$, we have $\openE\sigma \subseteq \openE$.
    Then $\openE\sigma = \openE$ because $\openE$ is finite.
    It now follows from $(*)$ that $\openE$ is closed under multiplication.
\end{proof}

The following result must be known but I have been unable to find a reference.

\begin{Theorem} \label{sf.3}
    Let $G \leq \gl{2}{\openF}$ and suppose that
    $G$ has an orbit $\Gamma$ on $\pl{\openF}$
    with the following properties.
    \begin{itemize}
        \item[(i)]  $G \cap \unip{\gamma}{\openF}$ is
                    transitive on $\Gamma \remove{\gamma}$
                    for some $\gamma \in \Gamma$.

        \item[(ii)] $\card{\Gamma} \geq 3$.

        \item[(iii)] If $\Char\openF = 2$ then
                     $G \cap \unip{\gamma}{\openF}$ is finite.
    \end{itemize}
    Then there exists a subfield $\openE \subseteq \openF$ and
    $h \in \gl{2}{\openF}$ such that the following hold.
    \begin{itemize}
        \item[(a)]  $\openE\additive \isom G \cap \unip{\gamma}{\openF}$.

        \item[(b)]  $\Gamma h = \pl{\openE}$.

        \item[(c)]  $\spl{2}{\openE} \normal G^{h} \leq \gl{2}{\openF}_{(\pl{\openE})}$.
    \end{itemize}
\end{Theorem}
\begin{proof}
    Note that (i) holds for all $\gamma \in \Gamma$ since
    $\Gamma$ is an orbit.
    Now $\card{\Gamma} \geq 3$ and
    $\gl{2}{\openF}$ is $3$-transitive on $\pl{\openF}$
    so conjugating $G$ by a suitable member of $\gl{2}{\openF}$
    we may suppose that $[1:0], [0:1], [1:1] \in \Gamma$.

    Let $\alpha = [1:0], \beta = [0:1]$ and
    define $\openE_{\alpha}, \openE_{\beta} \subseteq \openF$ by \[
        G \cap \unip{\alpha}{\openF} = \bbset{ \PMatrixC{1 & 0 \\ a & 1} }{ a \in \openE_{\alpha} } \text{\;and\;}
        G \cap \unip{\beta}{\openF} = \bbset{ \PMatrixC{1 & a \\ 0 & 1} }{ a \in \openE_{\beta} }.
    \]
    Then $\openE_{\alpha}$ and $\openE_{\beta}$ are subgroups of $\openF\additive$.
    We have
    \begin{align}
        \beta ( G \cap \unip{\alpha}{\openF} ) &= %
        \bbset{ [0:1]\PMatrixC{1 & 0 \\ a & 1} }{ a \in \openE_{\alpha} } = %
        \bset{ [a:1] }{ a \in \openE_{\alpha} }. \notag \\
    \intertext{Now $G \cap \unip{\alpha}{\openF}$ is transitive on $\Gamma \remove{\alpha}$ so}
        \Gamma \remove{\alpha, \beta} &=
        \bset{ [a:1] }{ a \in \openE_{\alpha} \remove{0} }. \tag{1} \\
    \intertext{Similarly}
    \alpha ( G \cap \unip{\beta}{\openF} ) &= %
        \bbset{ [1:0]\PMatrixC{1 & a \\ 0 & 1} }{ a \in \openE_{\beta} } =%
        \bset{ [1:a] }{ a \in \openE_{\beta} }. \notag \\
    \intertext{and}
        \Gamma \setminus \listset{\alpha, \beta} &=
        \bset{ [1:a] }{ a \in \openE_{\beta} \remove{0} } = %
        \bset{ [a^{-1}:1] }{ a \in \openE_{\beta} \remove{0} }. \tag{2}
    \intertext{From $(1)$ and $(2)$ we obtain}
        \openE_{\alpha} \remove{0} &=
        \bigl( \openE_{\beta} \remove{0} \bigr)^{-1}. \tag{3}
    \end{align}

     Since $[1:1] \in \Gamma$ it follows from $(1)$ and $(2)$ that
     $1 \in \openE_{\alpha} \cap \openE_{\beta}$.
     Then $-1 \in \openE_{\alpha}$ because $\openE_{\alpha} \leq \openF\additive$.
     Let \[
        w = \PMatrixC{ 0 & 1 \\ -1 & 0 } = %
        \PMatrixC{1 & 1 \\ 0 & 1}\PMatrix{1 & 0 \\ -1 & 1}\PMatrixC{1 & 1 \\ 0 & 1} \in G.
     \]
     Since $w \in G$ and $\alpha w = \beta$ we have \[
        ( G \cap \unip{\alpha}{\openF} )^{w} = G \cap \unip{\beta}{\openF}.
     \]
     For any $a \in \openF$ we also have \[
        \PMatrixC{1 & 0 \\ a & 1}^{w} = \PMatrixC{1 & -a \\ 0 & 1}.
     \]
     Consequently $\openE_{\beta} = -\openE_{\alpha} = \openE_{\alpha}$,
     again because $\openE_{\alpha} \leq \openF\additive$.

     Let $\openE = \openE_{\alpha}$.
     Then $(3)$ and Lemma~\ref{sf.2} imply that $\openE$ is a subfield of $\openF$.
     By the definition of $\openE_{\alpha}$ we have
     $\openE\additive \isom G \cap \unip{\alpha}{\openF}$.
     Recall that $\alpha = [1:0]$ and $\beta = [0:1]$.
     Then $(1)$ implies that \[
        \Gamma = \bset{ [a:1] }{ a \in \openE } \cup \blistset{ [1:0] } = \pl{\openE}.
     \]
     Now $\Gamma$ is an orbit of $G$ hence $G \leq \gl{2}{\openF}_{(\pl{\openE})}$.
     We have \[
        \bbgen{ \PMatrixC{1 & 0 \\ a & 1}, \PMatrixC{1 & a \\ 0 & 1} }{ a \in \openE } \leq G.
     \]
     The left hand side is well known to be $\spl{2}{\openE}$
     whence $\spl{2}{\openE} \leq G$ and the proof is complete.
\end{proof}

%% file: exc.tex
\section{The exceptional case} \label{exc}

Recall that if $\openF$ is a field with characteristic $3$ that
contains a square root $i$ of $-1$ then \[
    \afive{\openF} = \bblistgen{ \PMatrixC{1 & 0 \\ 1 & 1}, \PMatrixC{1 & i \\ 0 & 1} }
\]
and $\afiveBAR{\openF}$ is the image of $\afive{\openF}$ in $\pgl{2}{\openF}$.

The goal of this section is to study and characterize these groups.

\newpage

\begin{Lemma} \label{exc.1}
    Suppose that $\Char\openF = 3$ and that
    $\openF$ contains  a square root $i$ of $-1$.
    \begin{itemize}
        \item[(a)]  $\afiveBAR{\openF} \isom %
                    \afive{\openF} / \zenter{\afive{\openF}} \isom \alt{5}$.

        \item[(b)]  $\zenter{\afive{\openF}} = \listset{\pm 1}$ and
                    $\afive{\openF}$ is perfect.
    \end{itemize}
\end{Lemma}
\begin{proof}
    Without loss, $\openF$ is generated by $i$ so that $\card{\openF} = 9$.
    Let \[
        g = \PMatrixC{ 1 & 1 \\0 & 1}, \quad h = \PMatrixC{1 & i \\ 0 & 1}, \qtext{and} %
        t = \PMatrix{0 & -1 \\ 1 & 0}.
    \]
    By \cite[p.52]{W} there is
    an isomorphism $\theta : \psl{2}{\openF} \longrightarrow \alt{6}$ with \[
        \BAR{g}\theta = (123), \quad \BAR{h}\theta = (456) \qtext{and} \BAR{t}\theta = (23)(14).
    \]
    Let $\Omega = \listset{2,3,4,5,6}$ and note that $g^{t} = \PSmallMatrixC{1 & 0 \\ -1 & 1}$.
    Then \[
        \afiveBAR{\openF}\theta = \blistgen{ \BAR{g}^{\BAR{t}}, \BAR{h} }\theta =%
        \blistgen{ (432), (456) } = \alt{\Omega} \isom \alt{5}.
    \]
    The kernel of the map $\spl{2}{\openF} \longrightarrow \psl{2}{\openF}$
    sending $k$ to $\BAR{k}$ is $\listset{\pm 1}$ and
    this is the unique subgroup of order $2$ in $\spl{2}{\openF}$.
    Since $\alt{5}$ is simple,
    the conclusions follow.
\end{proof}

\begin{Lemma} \label{exc.2}
    Suppose that $\alt{5}$ is generated by two elements $a$ and $b$
    that have order $3$.
    Then $[a,b]$ is conjugate to $a$.
\end{Lemma}
\begin{proof}
    Without loss $a = (123)$ and $b = (345)$.
    Then \[
        [a,b] = (321)(543)(123)(345) = (134).
    \]
\end{proof}

\begin{Lemma} \label{exc.3}
    Suppose that $\Char\openF  = 3, l \in \openF \setminus \listset{0}$,
    set \[
        G = \bblistgen{ \PMatrixC{1 & 0 \\ 1 & 1}, \PMatrixC{ 1 & l \\ 0 & 1} }
    \]
    and let $\BAR{G}$ be the image of $G$ in $\pgl{2}{\openF}$.
    Assume that $\BAR{G} \isom \alt{5}$.
    Then $l^{2} = -1$ and $G = \afive{\openF}$.
\end{Lemma}
\begin{proof}
    Let $a = \PSmallMatrixC{1 & 0 \\ 1 & 1}$ and
    $b = \PSmallMatrixC{ 1 & l \\ 0 & 1}$.
    Now $\Char\openF = 3$ so $\BAR{a}$ and $\BAR{b}$ have order $3$.
    Since $\BAR{G} = \listgen{ \BAR{a}, \BAR{b} } \isom \alt{5}$,
    Lemma~\ref{exc.2} implies that
    $[\BAR{a}, \BAR{b}]$ is conjugate to $\BAR{a}$.
    Then $[a,b] = za^{g}$ for some
    $z \in \zenter{\gl{2}{\openF}}$ and $g \in \gl{2}{\openF}$.
    Now $a,b \in \spl{2}{\openF}$ so
    $z \in \zenter{\spl{2}{\openF}} = \listset{\pm 1}$.
    Hence \[
        \operatorname{trace}\,[a,b] = \pm \operatorname{trace}\,a = \pm 2.
    \]
    Now
    \begin{align*}
        [a,b] &= \PMatrixC{1 & 0 \\ -1 & 1}\PMatrixC{1 & -l \\ 0 & 1}%
        \PMatrixC{1 & 0 \\ 1 & 1}\PMatrixC{1 & l \\ 0 & 1}\\
        &= \PMatrixC{1 & -l \\ -1 & l+1}\PMatrixC{1 & l \\ 1 & l+1} = %
        \PMatrixC{ 1-l & * \\ * & l^{2} + l + 1 } \\
    \end{align*}
    so $l^{2} + 2 = \pm 2$.
    Since $l \not= 0$ and $\Char\openF = 3$ this forces $l^{2} = -1$
    and the conclusion follows from the definition of $\afive{\openF}$.
\end{proof}

\begin{Lemma} \label{exc.4}
    Suppose that $\Char\openF = 3$,
    that $\BAR{G} \leq \pgl{2}{\openF}$ and set  \[
        \Gamma = \bset{ \gamma \in \pl{\openF} }{\mbox{%
            $\listset{\gamma} = \fix{\Omega}{\BAR{g}}$ for some $\BAR{g} \in \BAR{G}\nonid$
        }}.
    \]
    Suppose also that $\BAR{G} \isom \alt{5}$.
    Then $\openF$ contains a square root $i$ of $-1$.
    Let $\openE$ be the subfield generated by $i$.
    Then $\card{\openE} = 9$ and there exists $h \in \gl{2}{\openF}$
    such that \[
        \BAR{G}^{\BAR{h}} = \afiveBAR{\openF} \qtext{and} \Gamma \BAR{h} = \pl{\openE}.
    \]
\end{Lemma}
\begin{proof}
    Choose $\BAR{a}, \BAR{b} \in \BAR{G}$ that
    map to the $3$-cycles $(123)$ and $(345)$.
    Then $\BAR{G} = \listgen{ \BAR{a}, \BAR{b} }$.
    Now $\alt{5}$ is perfect so $\BAR{G} \leq \psl{2}{\openF}$.
    Since $\zenter{\spl{2}{\openF}} = \listset{ \pm 1 }$
    there exist inverse images $a$ and $b$ of $\BAR{a}$ and $\BAR{b}$
    that have order $3$.
    By Lemma~\ref{pl.2} there exist $\alpha, \beta \in \pl{\openF}$
    with $a \in \unip{\alpha}{\openF}$ and $b \in \unip{\beta}{\openF}$.
    Then $\alpha \not= \beta$ because $[a,b] \not= 1$.
    By Lemma~\ref{pl.3} we have \[
        a^{h} = \PMatrixC{ 1 & 0 \\ 1 & 1} \qtext{and} b^{h} = \PMatrixC{ 1 & i \\ 0 & 1}
    \]
    for some $h \in \gl{2}{\openF}$ and $i \in \openF$.

    Let $G = \listgen{a,b}$.
    Then $\BAR{G}^{\BAR{h}} \isom \alt{5}$.
    The previous lemma implies that $i^{2} = -1$ and
    that $G^{h} = \afive{\openF}$.
    Then $\BAR{G}^{\BAR{h}} = \afiveBAR{\openF}$.

    To prove the final assertion we may suppose that $h = 1$
    so that $G = \afive{\openF}$ and
    $G \leq \spl{2}{\openE} \leq \spl{2}{\openF}$.

    Let $\gamma \in \Gamma$ so that
    $\listset{ \gamma } = \fix{\pl{\openF}}{\BAR{g}}$
    for some $\BAR{g} \in \BAR{G}\nonid$.
    Lemma~\ref{pl.2} implies that $\BAR{g}$ is a $3$-element and
    then that $\fix{\pl{\openE}}{\BAR{g}} \not= \emptyset$.
    Then $\gamma \in \pl{\openE}$ and
    we deduce that $\Gamma \subseteq \pl{\openE}$.

    Since $\BAR{G} \isom \alt{5}$ we may
    choose $\BAR{g} \in \BAR{G}$ with order $3$.
    By Lemma~\ref{pl.2} we have
    $\fix{\pl{\openF}}{\BAR{g}} = \listset{\gamma}$
    for some $\gamma$.
    Set $\BAR{U} = \listgen{\BAR{g}}$.
    Then $\BAR{U} \leq \BAR{G} \cap \unipBAR{\gamma}{\openE}$
    by Lemma~\ref{pl.1}(c).
    Since $\unipBAR{\gamma}{\openE}$ is a $3$-group and
    $\BAR{G} \isom \alt{5}$ we obtain \[
        \cyclic{3} \isom \BAR{U} = \BAR{G} \cap \unipBAR{\gamma}{\openE} \normal \BAR{G}_{\gamma},
    \]
    and then $\card{\BAR{G_{\gamma}}} = 3$ or $6$.
    Hence the orbit of $\BAR{G}$ on $\pl{\openF}$ that contains $\gamma$
    has length $20$ or $10$.
    Now $\card{\pl{\openE}} = \card{\openE} + 1 = 10$.
    Thus $\BAR{G}$ is transitive on $\pl{\openE}$
    and it follows that $\Gamma = \pl{\openE}$.
\end{proof}

%% file: ptab.tex
\section{The proofs of \thmA\ and \thmB} \label{ptab}

It is now just a matter of putting everything together.

\begin{proof}[Proof of \thmA]
    Let $p = \Char\openF$.
    We assume that every orbit of $\BAR{G}$ on $\pl{\extF}$
    has length at least $3$ since otherwise Conclusion~(a) holds.
    We have \[
        \card{\fix{\pl{\extF}}{\BAR{g}}} \in \listset{1,2}
    \]
    for all $\BAR{g} \in \BAR{G}\nonid$.

    Suppose that $\card{\fix{\pl{\extF}}{\BAR{g}}} = 2$
    for all $\BAR{g} \in \BAR{G}\nonid$.
    We claim that $p$ does not divide $\card{\BAR{G}}$.
    If $p = 0$ this is clear.
    Suppose $p > 0$.
    Then $\openF\mult$ does not contain any elements of order $p$.
    Lemma~\ref{pl.1}(c)(ii) implies the same for $\BAR{G}$,
    so the claim holds in this case also.
    \thmC\ now implies that Conclusion~(b) holds.

    Suppose that $\card{\fix{\pl{\extF}}{\BAR{g}}} = 1$
    for some $\BAR{g} \in \BAR{G}\nonid$.
    Lemma~\ref{pl.1}(c) implies that $\listgen{\BAR{g}}$ is isomorphic
    to a subgroup of $\openF\additive$.
    Now $\BAR{G}$ is finite so $p > 0$ and $p$ divides $\BAR{G}$.
    Recall that \[
        \Gamma = \bset{ \gamma \in \pl{\extF} }{\mbox{%
            $\listset{\gamma} = \fix{\Omega}{\BAR{g}}$ for some $\BAR{g} \in \BAR{G}\nonid$
        }}.
    \]
    For each $\gamma \in \Gamma$ let \[
        \BAR{K}_{\gamma} = \bset{ \BAR{g} \in \BAR{G}_{\gamma}\nonid }{%
            \fix{ \pl{\extF} \remove{\gamma} }{\BAR{g}} = \emptyset
        } \cup \blistset{1}.
    \]
    Lemma~\ref{pl.1}(c) implies that \[
        \BAR{K}_{\gamma} = \BAR{G} \cap \unipBAR{\gamma}{\extF}.
    \]
    In particular, $\BAR{K}_{\gamma}$ is a subgroup.

    We apply \thmD.
    Then $\Gamma$ is an orbit.
    Fix $\gamma \in \Gamma$ and let
    $q = \card{\BAR{K}_{\gamma}} = \card{ \BAR{G} \cap \unipBAR{\gamma}{\extF} }$.
    Suppose that Conclusion~(b) of \thmD\ holds.
    Then $\BAR{G} \isom \alt{5}$ and $q = 3$.
    Since $\unipBAR{\gamma}{\extF} \isom \extF\additive$
    this forces $p = 3$.
    Lemma~\ref{exc.4} implies that $\BAR{G}$ is
    conjugate in $\pgl{2}{\openF}$ to $\afiveBAR{\openF}$ and
    so Conclusion~(c)(ii) of \thmA\ holds.
    Hence we assume that Conclusion~(a) of \thmD\ holds.

    Consider the action of $\BAR{K}_{\gamma}$ on $\Gamma \remove{\gamma}$.
    By the definition of $\BAR{K}_{\gamma}$,
    this action is semiregular.
    Now $\card{\BAR{K}_{\gamma}} = q = \card{\Gamma \remove{\gamma}}$
    so it is regular.
    In particular \[ \mbox{%
        $\BAR{G} \cap \unipBAR{\gamma}{\extF}$ is %
        transitive on $\Gamma \remove{\gamma}$ %
        and $q \geq 3$.
    }\]

    Let $G$ be the full inverse image of $\BAR{G}$ in $\gl{2}{\extF}$.
    Recall that $\unip{\gamma}{\extF}$ maps
    isomorphically onto $\unipBAR{\gamma}{\extF}$.
    Then $\card{ G \cap \unip{\gamma}{\extF} } = %
    \card{ \BAR{G} \cap \unipBAR{\gamma}{\extF} } = q$ and
    $G \cap \unip{\gamma}{\extF}$ is transitive on $\Gamma \remove{\gamma}$.

    Theorem~\ref{sf.3} implies there exists a subfield $\openE \subseteq \extF$
    and $h \in \gl{2}{\extF}$ such that
    \begin{itemize}
        \item   $\openE\additive \isom G \cap \unip{\gamma}{\extF}$, so $\card{\openE} = q$;

        \item   $\Gamma h = \pl{\openE}$; and

        \item   $\psl{2}{\openE} \normal \BAR{G}^{\BAR{h}} \leq \pgl{2}{\extF}_{(\pl{\openE})}$.
    \end{itemize}
    Now $\pgl{2}{\extF}_{(\pl{\openE})} = \pgl{2}{\openE}$ by Lemma~\ref{sf.1}.
    Also  \mbox{$\card{ \pgl{2}{\openE} : \psl{2}{\openE} } \leq 2$} because $\openE$ is finite.
    Consequently \[
        \BAR{G}^{\BAR{h}} = \psl{2}{\openE} \qtext{or} \pgl{2}{\openE}.
    \]

    Suppose, for a contradiction, that $\Gamma \nsubseteq \pl{\openF}$.
    Choose $\alpha \in \Gamma \setminus \pl{\openF}$ and $\BAR{g} \in \BAR{G}\nonid$
    with $\listset{\alpha} = \fix{\pl{\extF}}{\BAR{g}}$.
    Then $\fix{\pl{\openF}}{\BAR{g}} = \emptyset$.
    Lemma~\ref{pl.1} implies that $\BAR{g} \in \unipBAR{\alpha}{\extF} \isom \openF\additive$
    so $\BAR{g}$ is a $p$-element.
    By hypothesis, $\BAR{G} \leq \pgl{2}{\openF}$.
    Lemma~\ref{pl.2}(b),(c) implies that
    \begin{equation} \tag{$*$}
        \BAR{g} \not\in \psl{2}{\openF} \qtext{and} p = 2.
    \end{equation}
    Now $\openE$ is finite with characteristic $2$ so
    $\psl{2}{\openE} = \pgl{2}{\openE}$.
    Moreover $\card{\openE} = q \geq 3$ so $\card{\openE} \geq 4$ and
    $\psl{2}{\openE}$ is perfect.
    Then so is $\BAR{G}$.
    Since $\pgl{2}{\openF} / \psl{2}{\openF}$ is abelian we obtain
    $\BAR{G} \leq \psl{2}{\openF}$,
    contrary to $(*)$.
    We deduce that \[
        \Gamma \subseteq \pl{\openF}.
    \]
    We have $\pl{\openE} = \Gamma\BAR{h} \subseteq \pl{\openF}\BAR{h}$.
    Now \[
        [1:0], [0:1], [1:1] \in \pl{\openE} \cap \pl{\openF} \subseteq %
        \pl{\openF}\BAR{h} \cap \pl{\openF}.
    \]
    Lemma~\ref{sf.1}(a) forces $\BAR{h} \in \pgl{2}{\openF}$.
    Then $\pl{\openE} \subseteq \pl{\openF}$.
    Consequently $\openE \subseteq \openF$ and Conclusion~(c)(i) holds.
\end{proof}

\begin{Lemma} \label{ptab.1}
    Suppose that $\Char\openF$ is odd,
    that $G \leq \gl{2}{\openF}$,
    that $\BAR{G}$ is finite and that \[
        \PMatrixC{1 & 0 \\ 1 & 1}, \PMatrixC{1 & l \\ 0 & 1} \in G
    \]
    for some $l \in \openF \remove{0}$.
    Then there exists a subfield $\openE \subseteq \openF$ with $l \in \openE$
    such that one of the following holds.
    \begin{itemize}
        \item[(a)]  $\BAR{G} = \psl{2}{\openE}$ or $\pgl{2}{\openE}$.

        \item[(b)]  $\Char\openF = 3$, $\openF$ contains a square root of $-1$,
                    $\card{\openE} = 9$ and
                    $\BAR{G}$ is conjugate to $\afiveBAR{\openF}$.
    \end{itemize}
\end{Lemma}
\begin{proof}
    Let $\extF$ be an algebraic closure of $\openF$.
    The subgroups generated by $\PSmallMatrixC{1 & 0 \\ 1 & 1}$ and
    $\PSmallMatrixC{1 & l \\ 0 & 1}$ have order $\Char\openF \geq 3$
    and act semiregularly on
    $\pl{\extF} \remove{[1:0]}$ and
    $\pl{\extF} \remove{[0:1]}$ respectively.
    Consequently every orbit of $\BAR{G}$ on $\pl{\extF}$ has
    length at least three and
    $\BAR{G}$ has order divisible by $\Char\openF$.
    Conclusion~(c) of \thmA\ holds.
    Set \[
        \Gamma = \bset{ \gamma \in \pl{\extF} }{\mbox{%
            $\listset{\gamma} = \fix{\Omega}{\BAR{g}}$ for some $\BAR{g} \in \BAR{G}\nonid$
        }}.
    \]
    Then $\Gamma$ is an orbit and there exists a subfield $\openE \subseteq \openF$ and
    $\BAR{h} \in \pgl{2}{\openF}$ such that \[
        \Gamma\BAR{h} = \pl{\openE}.
    \]
    Now $[1:0]$ and $[0:1]$ are the unique fixed points of
    $\PSmallMatrixC{1 & 0 \\ 1 & 1}$ and
    $\PSmallMatrixC{1 & l \\ 0 & 1}$ respectively and
    $[1:1] = [0:1]\PSmallMatrixC{1 & 0 \\ 1 & 1}$.
    Then \[
        [1:0], [0:1], [1:1] \in \Gamma \cap \pl{\openE}  = \pl{\openE}\BAR{h}^{-1} \cap \pl{\openE}.
    \]
    Lemma~\ref{sf.1}(a) forces $\BAR{h} \in \pgl{2}{\openE}$ whence \[
        \Gamma = \pl{\openE}.
    \]
    Moreover $[l^{-1}:1] = [1:l] = [1:0]\PSmallMatrixC{1 & l \\ 0 & 1} \in \Gamma$
    whence $l \in \openE$.

    Suppose that Conclusion~(c)(i) of \thmA\ holds.
    Since $\BAR{h} \in \pgl{2}{\openE}$ we have
    $\BAR{G} = \psl{2}{\openE}$ or $\pgl{2}{\openE}$ and (a) holds.
    If Conclusion~(c)(ii) holds then (b) holds.
\end{proof}

\begin{proof}[Proof of \thmB]
    We apply Lemma~\ref{ptab.1}.
    Now $l$ generates $\openF$ so $\openE = \openF$.
    Suppose that Conclusion~(a) of Lemma~\ref{ptab.1} holds.
    Now $G \leq \spl{2}{\openF}$ whence $\BAR{G} = \psl{2}{\openF}$.
    Since $\listset{\pm 1}$ is the unique subgroup of order $2$
    in $\spl{2}{\openF}$ it follows that $G = \spl{2}{\openF}$
    so (a) holds.
    Suppose that Conclusion~(b) of Lemma~\ref{ptab.1} holds.
    Then $\BAR{G} \isom \alt{5}$ and (b) holds by Lemmas~\ref{exc.3} and \ref{exc.1}.
\end{proof}

%% file: pte.tex
\section{The proof of \thmE} \label{pte}
We have the finite group $G$ acting on the set $\Omega$
with the property
\begin{equation} \tag{$*$}
    \mbox{\em $\card{\fix{\Omega}{g}} \in \listset{1,2}$ for all $g \in G\nonid$.}
\end{equation}
In particular \[\mbox{\em
    every $3$-point stabilizer is trivial.
}\]

We assume that each orbit has length at least three
since otherwise Conclusion~(a) holds.
If $\card{\fix{\Omega}{g}} = 2$ for all $g \in G\nonid$
then \thmC\ implies that Conclusion~(b) holds.
Hence we assume that $\card{\fix{\Omega}{g}} = 1$ for some $g \in G\nonid$.
The hypotheses of \thmD\ are satisfied.
If Conclusion~(b) of \thmD\ holds then Conclusion~(d) holds.
Hence we assume that Conclusion~(a) of \thmD\ holds.

A special case of a theorem of Zassenhaus,
Theorem~\ref{zg.4} of the appendix,
will be used to show that Conclusion~(c) holds.
Let \[
    \Gamma = \bset{ \gamma \in \Omega }{\mbox{%
            $\listset{\gamma} = \fix{\Omega}{g}$ for some $g \in G\nonid$
    }}.
\]
\thmD\ implies that \[\mbox{\em
    $G$ acts $2$-transitively on $\Gamma$.
}\]
In particular, $\Gamma$ is an orbit so
$\card{\Gamma} \geq 3$ and $(*)$ implies that
$G$ acts faithfully on $\Gamma$.
Write $\card{\Gamma} = q+1$ with $q \in \openN$.
Let $\gamma, \gamma' \in \Gamma$ be distinct.
By \thmD\, \[
    \card{G_{\gamma\gamma'}} = \half(q-1) \qtext{\em or} q-1.
\]
We claim that \[\mbox{\em
    every element of $G_{\listset{\gamma,\gamma'}} \setminus G_{\gamma\gamma'}$
    is an involution.
}\]
Indeed suppose $g \in G_{\listset{\gamma,\gamma'}} \setminus G_{\gamma\gamma'}$.
By $(*)$ there exists $\alpha \in \fix{\Omega}{g}$.
Now $g$ interchanges $\gamma$ with $\gamma'$ so
$\alpha, \gamma$ and $\gamma'$ are distinct.
Moreover $g^{2}$ fixes $\alpha, \gamma$ and $\gamma'$
so $(*)$ forces $g^{2} = 1$.
The hypotheses of Theorem~\ref{zg.4} are satisfied
and the proof is complete.

%% file: zg.tex
\section{Appendix -- A theorem of Zassenhaus} \label{zg}

{
\newcommand{\bfZero}{\mathbf{0}}
\newcommand{\bfOne}{\mathbf{1}}

For the proof of \thmE\ we need a special case of
a theorem of Zassenhaus that characterizes the groups
$\psl{2}{\openF}$ and $\pgl{2}{\openF}$.
This section consists of a presentation of selected material
from \cite[pp.161--172]{HBIII}.

Given a field $\openF$ and a symbol $\infty \not\in \openF$
we define \[
    \openF^{\infty} = \openF \cup \listset{ \infty }
\]
and extend arithmetic to $\openF^{\infty}$ in the obvious way.
Recall that $\pl{\openF}$,
the projective line over $\openF$ satisfies \[
    \pl{\openF} = \set{ [x:1] }{ x \in \openF } \cup \listset{ [1:0] }.
\]
Thus we identify $\openF^{\infty}$ with $\pl{\openF}$ via \[\mbox{
    $x \mapsto [x:1]$ for $x \in \openF$ and $\infty \mapsto [1:0]$.
}\]
In this way,
a matrix $\PSmallMatrixC{a & b \\ c & d} \in \gl{2}{\openF}$
corresponds to the \emph{linear fractional transformation}
$\openF^{\infty} \longrightarrow \openF^{\infty}$ given by \[
    x \mapsto \frac{ax+c}{bx+d}.
\]
Hence we regard $\pgl{2}{\openF}$ and $\psl{2}{\openF}$ as
permutation groups on $\openF^{\infty}$.
We have
\begin{align*}
    \pgl{2}{\openF} &= \set{ x \mapsto \frac{ax+c}{bx+d} }{ a,b,c,d \in \openF, ad - bc \not= 0}
\intertext{and}
    \psl{2}{\openF} &= \set{ x \mapsto \frac{ax+c}{bx+d} }{\mbox{
            $a,b,c,d \in \openF$ and $ad - bc$ is a nonzero square%
            }}.
\intertext{Then}
    \pgl{2}{\openF}_{\infty} &= \set{ x \mapsto ax+c }{ a \in \openF \remove{0}, c \in \openF}
\intertext{and}
    \psl{2}{\openF}_{\infty} &= \set{ x \mapsto ax+c }{\mbox{
            $a \in \openF \remove{0}, c \in \openF$ and $a$ is a square%
            }}.
\end{align*}
Note also that the matrix $\PSmallMatrixC{0 & 1 \\-1 & 0} \in \spl{2}{\openF}$
gives the transformation $x \mapsto -1/x$ in $\psl{2}{\openF}$.

\begin{Lemma} \label{zg.1}
    Let the group $M$ act on the set $\openF$.
    Suppose that
    \begin{itemize}
        \item[(i)]  $K$ is an abelian minimal normal subgroup of $M$;

        \item[(ii)] $K$ is regular on $\openF$; and

        \item[(iii)] $M_{\alpha}$ is abelian for all $\alpha \in \openF$.
    \end{itemize}
    Let $\bfZero, \bfOne \in \openF$ be distinct.
    Then there exist binary operations $+$ and $\cdot$ on $\openF$
    such that the following hold.
    \begin{itemize}
        \item[(a)]  $(\openF, +, \cdot, \bfZero, \bfOne)$ is a field.

        \item[(b)]  The map $k \mapsto \bfZero k$ is an isomorphism $K \longrightarrow \openF\additive$.

        \item[(c)]  the map $h \mapsto \bfOne h$ is a monomorphism $M_{\bfZero} \longrightarrow \openF\mult$.

        \item[(d)]  For all $x \in \openF, k \in K$ and $h \in M_{\bfZero}$, \[
            xk = x + ( \bfZero k) \qtext{and} xh = x \cdot ( \bfOne h ).
        \]
    \end{itemize}
\end{Lemma}
\begin{proof}
    Let $H = M_{\bfZero}$.
    Now $\fix{\openF}{\cc{H}{K}}$ is $K$-invariant and contains $\bfZero$
    so by (ii), $\fix{\openF}{\cc{H}{K}} = \openF$.
    Then $\cc{H}{K} = 1$ and $H$ acts faithfully
    on $K$ by conjugation.
    Since $K$ is abelian,
    $\EEnd{K}$ is a ring and we identify $H$ with its image in
    the group of units of $\EEnd{K}$.
    We use the exponential notation $k^{e}$ to denote the
    image of an element $k \in K$ under $e \in \EEnd{K}$.

    Let $\calF$ be the subring of $\EEnd{K}$ generated by $H$.
    Then $\calF$ is commutative since $H$ is abelian.
    By~(i), $1_{K}$ and $K$ are the only
    $\calF$-invariant subgroups of $K$.
    If $f \in \calF$ then the kernel and image of $f$ are $\calF$-invariant.
    It follows that $\calF$ is a field,
    that $K$ is an $\calF$-vectorspace and
    then that $\dim_{\calF}K = 1$.

    Define $\pi : K \longrightarrow \openF$ by $k\pi = \bfZero k$.
    By (ii), $\pi$ is a bijection.
    Choose $k_{1} \in K$ with $\bfOne = k_{1}\pi$.
    Then $k_{1} \not= 1_{K}$ since $\bfOne \not= \bfZero$.
    Define $\rho : \calF \longrightarrow K$ by $f\rho = k_{1}^{f}$.
    Since $\dim_{\calF}K = 1$ it follows that $\rho$ is
    an isomorphism $\calF\additive \longrightarrow K$.
    Let $\sigma = \rho\pi$.
    \[
        \sigma : f \xmapsto{\;\;\;\;\rho\;\;\;\;} k_{1}^{f} %
        \xmapsto{\;\;\;\;\pi\;\;\;\;} \bfZero k_{1}^{f}.
    \]
    Then $\sigma$ is a bijection $\calF \longrightarrow \openF$.
    Since $\calF$ is a field there are binary operations $+$ and $\cdot$
    on $\openF$ which endow $\openF$ with the structure of a field
    in such a way that $\sigma$ is a field isomorphism.
    We have \[
        0_{\calF}\sigma = \bfZero k_{1}^{0_{\calF}} = \bfZero 1_{K} = \bfZero \qtext{and}%
        1_{\calF}\sigma = \bfZero k_{1}^{1_{\calF}} = \bfZero k_{1} = \bfOne.
    \]
    Then (a) holds.

    Now $\pi = \rho^{-1}\sigma$ so since $\rho^{-1} : K \longrightarrow \calF\additive$
    is an isomorphism it follows that $\pi$ is
    an isomorphism $K \longrightarrow \openF\additive$,
    proving (b).
    Let $x \in \openF$ and $k \in K$.
    Now $x = \bfZero l$ for some $l \in K$ whence \[
        xk = (\bfZero l)k = \bfZero (lk) = \bfZero l + \bfZero k = x + \bfZero k,
    \]
    proving the first part of (d).

    Let $f \in \calF$ and $h \in H$.
    Note that $\bfZero h^{-1} = \bfZero$ since $H = M_{\bfZero}$.
    Then \[
        (f\sigma) \cdot (h\sigma) = (fh)\sigma = \bfZero k_{1}^{fh} = %
        \bfZero h^{-1}k_{1}^{f}h = \bfZero k_{1}^{f}h = (f\sigma)h.
    \]
    Putting $f = 1_{\calF}$ and recalling that $1_{\calF}\sigma = \bfOne$
    we obtain $h\sigma = \bfOne h$.
    Now $H$ is a subgroup of $\calF\mult$ and $\sigma$ is a field isomorphism
    so (c) holds.
    If $x \in \openF$ then $x = f\sigma$ for some $f \in \calF$ and the
    remaining assertion in (d) also follows.
\end{proof}

\begin{Lemma} \label{zg.2}
    Let the finite group $M$ act faithfully and transitively
    on the set $\openF$ of cardinality $q$.
    Assume that \[
        \card{ \fix{\openF}{m} } \in \listset{0, 1} \qtext{and}%
         \card{ M_{\alpha} } = \half(q-1) \text{\;\;or\;\;} q-1
    \]
    for all $m \in M\nonid$ and $\alpha \in \openF$.
    Let \[
        K = \set{  k \in M }{ \fix{\openF}{k} = \emptyset } \cup \listset{ 1 }.
    \]
    Then $K$ is an elementary abelian minimal normal subgroup of $M$,
    $K$ is regular on $\openF$ and $M = M_{\alpha}K$.
\end{Lemma}
\begin{proof}
    First we claim that if $k \in K\nonid$ then $\cc{M}{k} = \cc{K}{k}$.
    Assume false and choose $m \in \cc{M}{k} \setminus K$.
    Then $\fix{\openF}{m} = \listset{\beta}$ for some $\beta$.
    But then $k$ fixes $\beta$.
    This contradiction proves the claim.
    In particular,
    $\cc{K}{k}$ is a subgroup.

    Lemma~\ref{stab.1} implies that $M$ has one orbit on $\openF$ of length $\card{K}$
    and any remaining orbits are regular.
    Since $M$ is transitive we obtain $\card{K}= \card{\openF} = q$.
    Let $\alpha \in \openF$.
    Then $\card{M} = \card{M_{\alpha}}\card{\openF} = \card{M_{\alpha}}\card{K}$.
    Note that $K$ is a normal subset of $M$.
    Let $k_{1}, \ldots, k_{n}$ be representatives for
    the conjugacy classes of $M$ contained in $K\nonid$.
    Using the claim we have \[
        \card{K} - 1 = \sum_{i=1}^{n} \card{M:\cc{M}{k_{i}}} = %
        \card{M_{\alpha}}\sum_{i=1}^{n} \frac{\card{K}}{\card{\cc{K}{k_{i}}}}
    \]
    so \[
        \frac{q-1}{\card{M_{\alpha}}} = \sum_{i=1}^{n} \frac{\card{K}}{\card{\cc{K}{k_{i}}}}.
    \]
    There are three possibilities.
    \begin{itemize}
        \item[(1)] $\card{M_{\alpha}} = q-1, n=1$ and $K = \cc{K}{k_{1}}$.

        \item[(2)] $\card{M_{\alpha}} = \half(q-1), n=1$ and $\card{K} = 2\card{\cc{K}{k_{1}}}$.

        \item[(3)] $\card{M_{\alpha}} = \half(q-1), n=2$ and $K = \cc{K}{k_{1}} = \cc{K}{k_{2}}$.
    \end{itemize}

    Recall that $\card{K} = q$.
    If $(2)$ holds then $q$ is even,
    contrary to $\card{M_{\alpha}} = \half(q-1)$.
    Hence $(1)$ or $(3)$ holds.
    Then $K$ is a subgroup and hence a normal subgroup of $M$.
    Suppose $L \normal M$ with $1 < L \leq K$.
    Without loss $k_{1} \in L$ so \[
        \card{L} \geq 1 + \half(q-1) = \half(q+1) > \half\card{K}
    \]
    so $L = K$.
    Hence $K$ is a minimal normal subgroup of $M$.
    Moreover, $k_{1} \in \zenter{K}$ so it follows that $K$ is elementary abelian.
    By its definition,
    $K$ is semiregular on $\openF$.
    Now $\card{K} = \card{\openF}$ so $K$ is regular on $\openF$.
\end{proof}

\begin{Lemma} \label{zg.3}
    Let $N$ be a proper normal subgroup of the group $G$ and suppose that
    every element of $G \setminus N$ is an involution.
    Then $N$ is abelian and
    $g^{t} = g^{-1}$ for all $g \in N$ and $t \in G \setminus N$.
\end{Lemma}
\begin{proof}
    Let $t \in G \setminus N$ and $g \in N$.
    Then $tg \in G \setminus N$ so
    $t^{2} = (tg)^{2} = 1$.
    Hence $t^{-1} = t, tgtg = 1$ and $g^{t} = g^{-1}$.
    Let $h \in N$.
    Then \[
        h^{-1}g^{-1} = (gh)^{-1} = (gh)^{t} = g^{t}h^{t} = g^{-1}h^{-1}
    \]
    and it follows that $N$ is abelian.
\end{proof}

The following is the special case of the theorem of Zassenhaus referred to at the beginning of this section.

\begin{Theorem} \label{zg.4}
    Let the group $G$ act on the set $\Omega$ of finite cardinality $q+1$.
    Let $\infty, \bfZero$ and $\bfOne$ be distinct elements of $\Omega$.
    Assume the following.
    \begin{itemize}
        \item[(i)]  $G$ is $2$-transitive and every $3$-point stabilizer is trivial.

        \item[(ii)] $\card{G_{\infty\bfZero}} = \half(q-1)$ or $q-1$.

        \item[(iii)]    Every element of $G_{\listset{\infty,\bfZero}} \setminus G_{\infty\bfZero}$
                        is an involution.
    \end{itemize}
    Let $\openF = \Omega \remove{\infty}$ so that $\Omega = \openF^{\infty}$.
    Then there exist binary operations $+$ and $\cdot$ on $\openF$ such that
    $(\openF, +, \cdot, \bfZero, \bfOne)$ is a field and
    the elements of $G$ are linear fractional transformation of $\openF^{\infty}$.
    Moreover $G = \pgl{2}{\openF}$ or $\psl{2}{\openF}$.
\end{Theorem}
\begin{proof}
    Let $M = G_{\infty}, H = G_{\infty\bfZero}, D = G_{\listset{\infty,\bfZero}}$
    and \[
        K = \set{  k \in M }{ \fix{\openF}{k} = \emptyset } \cup \listset{ 1 }.
    \]
   \setcounter{Step}{0}
    \begin{Step} 
        \begin{itemize}
            \item[(a)]  $H$ is an abelian normal subgroup of $D$ with index $2$
                        and $h^{t} = h^{-1}$ for all $h \in H$ and $t \in D \setminus H$.

            \item[(b)]  $K$ is an abelian minimal normal subgroup of $M$,
                        $K$ is regular on $\openF$ and $M = HK$.
        \end{itemize}
    \end{Step}
    \begin{proof}
        (a). Clearly $H \normal D$ and $\card{D:H} \leq 2$.
        By (i),
        $G$ contains an element that interchanges $\infty$ with $\bfZero$
        whence $\card{D:H} = 2$.
        Apply (iii) and Lemma~\ref{zg.3}

        (b). Note that $M$ acts on $\openF$ and by (i),
        $\card{\fix{\openF}{m}} \in \listset{0,1}$ for all $m \in M\nonid$.
        Apply Lemma~\ref{zg.2}
    \end{proof}

    The hypotheses of Lemma~\ref{zg.1} are satisfied and we adopt the notation
    defined in the conclusion.
    In particular,
    $(\openF, +, \cdot, \bfZero, \bfOne)$ is a field of order $q$.
    Let \[
        \calQ = \set{ \bfOne h }{ h \in H}.
    \]

    \begin{Step} 
        \begin{itemize}
            \item[(a)]  The set $\calQ$ is a subgroup of $\openF\mult$
                        that is isomorphic to $H$.

            \item[(b)]  The elements of $M$ are
                        all the linear fractional transformations of the form \[
                            x \mapsto a \cdot x + c
                        \]
                        with $a \in \calQ$ and $c \in \openF$.
        \end{itemize}
    \end{Step}
    \begin{proof}
        (a) follows from Lemma~\ref{zg.1}(c).
        To prove (b) recall that $M = HK$ and apply Lemma~\ref{zg.1}(b),(c),(d).
    \end{proof}

    \begin{Step} 
        Let $t \in D \setminus H$.
        The quantity \[
            x \cdot (xt)
        \]
        is constant as $x$ ranges over any orbit of $H$ on $\openF \remove{\bfZero}$.
    \end{Step}
    \begin{proof}
        Let $x \in \openF \remove{\bfZero}$ and $h \in H$.
        Lemma~\ref{zg.1}(d) implies that \[
            xh = x \cdot (\bfOne h).
        \]
        Step~1(a) and Lemma~\ref{zg.1}(d)(c) yield \[
            xht = xth^{-1} = (xt) \cdot (\bfOne h^{-1}) = (xt) \cdot (\bfOne h)^{-1}.
        \]
        Then $(xh) \cdot (xht) =%
        x \cdot (\bfOne h) \cdot (xt) \cdot (\bfOne h)^{-1} = x \cdot (xt)$.
    \end{proof}

    \begin{Step} 
        Suppose that $\card{H} = q - 1$.
        Then $G = \pgl{2}{\openF}$.
    \end{Step}
    \begin{proof}
        Recall that $\card{\calF} = q$.
        Step~2 implies that $\calQ = \openF \remove{\bfZero}$ and then that
        \begin{equation} \tag{$*$}
            M = \pgl{2}{\openF}_{\infty}.
        \end{equation}
        Now $\calQ$ is the $H$-orbit containing $\bfOne$ so
        $H$ is transitive on $\openF \remove{\bfZero}$.
        Choose $t \in D \setminus H$.
        Step~3 implies that
        $xt = (\bfOne t) \cdot x^{-1}$ for all $x \in \openF \remove{\bfZero}$.
        Since $t$ interchanges $\infty$ with $\bfZero$ it follows that $t$ is
        the linear fractional transformation \[
            x \mapsto (\bfOne t)/x.
        \]
        In particular, $t \in \pgl{2}{\openF}$.

        By $2$-transitivity,
        $M$ is a maximal subgroup of $G$ so $G = \listgen{M, t}$.
        For the same reason,
        $\pgl{2}{\openF} = \listgen{ \pgl{2}{\openF}_{\infty}, t }$.
        Then $(*)$ implies that $G = \pgl{2}{\openF}$.
    \end{proof}

    Henceforth we suppose that \[
        \card{H} = \half(q-1).
    \]
    then $q$ is odd.
    Let \[
        \calN = \bigl( \openF \remove{\bfZero} \bigr) \setminus \calQ.
    \]

    \begin{Step} 
        \begin{itemize}
            \item[(a)]  $\calQ$ and $\calN$ are
                        the orbits of $H$ on $\calF \remove{\bfZero}$ and
                        $\listset{ \calQ, \calN }$ is $D$-invariant.

            \item[(b)]  $M = \psl{2}{\openF}_{\infty}$.
        \end{itemize}
    \end{Step}
    \begin{proof}
        (a). Recall that $H = G_{\infty\bfZero}$ so $H$ is semiregular on $\calF \remove{\bfZero}$.
        Now $\card{\calQ} = \card{\calN} = \card{H}$ so it follows that
        $\calQ$ and $\calN$ are the orbits of $H$ on $\openF \remove{\bfZero}$.
        Since $H \normal D$ and $\openF \remove{\bfZero}$ is $D$-invariant
        it follows that $\listset{ \calQ, \calN }$ is $D$-invariant.

        (b). Since $\openF\mult$ is cyclic of order $q-1$ and $q$ is odd
        it follows that the set of squares in $\openF\mult$ is
        the unique subgroup of $\openF\mult$ with order $\half(q-1)$.
        Step~2(a) implies that $\calQ$ is the set of squares in $\calF\mult$
        and then Step~2(b) implies that $M = \psl{2}{\openF}_{\infty}$.
    \end{proof}

    \begin{Step} 
        Let $t \in D \setminus H$.
        \begin{itemize}
            \item[(a)]  If $q \equiv 1 \bmod 4$ then
                        $-\bfOne \in \calQ$ and $(\calQ t, \calN t) = (\calQ, \calN)$.

            \item[(b)]  If $q \equiv 3 \bmod 4$ then
                        $-\bfOne \in \calN$ and $(\calQ t, \calN t) = (\calN, \calQ)$.
        \end{itemize}
    \end{Step}
    \begin{proof}
        Since $\calF\mult$ is cyclic it follows that
        $-\bfOne$ is the unique involution in $\calF\mult$.

        Suppose that $q \equiv 1 \bmod 4$.
        Now $\calQ$ has order $\half(q-1)$,
        which is even,
        so $-\bfOne \in \calQ$.
        Also $\card{H} = \half(q-1)$.
        Let $r \in H$ be an involution.
        Then $\fix{\openF^{\infty}}{r} = \listset{ \infty, \bfZero }$.
        Now $r$ interchanges two points in $\openF^{\infty} \remove{\infty, \bfZero}$
        so by $2$-transitivity there is a conjugate $s$ of $r$ that
        interchanges $\infty$ with $\bfZero$.
        Then $s \in D$ and moreover $\card{\fix{\openF \remove{\bfZero}}{s}} = 2$.
        Step~5(a) forces $(\calQ s, \calN s) = (\calQ, \calN)$.
        As $\card{D:H} = 2$ we have $D = \listgen{H,s}$.
        Since $\calQ$ and $\calN$ are $H$-invariant and $t \in D \setminus H$
        it follows that $(\calQ t, \calN t) = (\calQ, \calN)$.
        Thus (a) holds.

        Suppose that $q \equiv 3 \bmod 4$.
        Now $\calQ$ has order $\half(q-1)$,
        which is odd,
        so $-\bfOne \in \calN$.
        Moreover $\card{H} = \half(q-1)$ so
        by $2$-transitivity it follows that every $2$-point stabilizer has odd order.
        By hypothesis~(iii), $t$ is an involution.
        Suppose that $(\calQ t, \calN t) = (\calQ, \calN)$.
        Since $\card{\calQ} = \card{\calN} = \half(q-1)$ it follows that
        $t$ has fixed points in $\calQ$ and $\calN$,
        a contradiction.
        Thus $(\calQ t, \calN t) = (\calN, \calQ)$ and (b) holds.
    \end{proof}

    \begin{Step} 
        There exists
        an involution $t \in D \setminus H$ and $y \in \openF \remove{\bfZero}$
        such that \[\mbox{
            $xt = (-\bfOne) \cdot x^{-1}$ for all $x \in \calQ$ and
            $xt = y \cdot x^{-1}$ for all $x \in \calN$.
        }\]
    \end{Step}
    \begin{proof}
        Choose $s \in D \setminus H$.
        Now $\bfOne \in \calQ$ so Step~6 implies that
        $\bfOne s$ and $-\bfOne$ lie in the same $H$-orbit.
        Choose $h \in H$ with $\bfOne sh = -\bfOne$ and set $t = sh$,
        so that $\bfOne t = -\bfOne$.
        Then $t \in D \setminus H$ so by hypothesis~(iii),
        $t$ is an involution.

        Step~3 implies that $x \cdot (xt)$ takes the constant value $\bfOne \cdot \bfOne t = -\bfOne$
        as $x$ ranges over $\calQ$.
        Hence $xt = -\bfOne \cdot x^{-1}$ for all $x \in \calQ$.

        Similarly $x \cdot (xt)$ takes some constant value $y$ as $x$ varies over $\calN$.
        Then $xt = y \cdot x^{-1}$ for all $x \in \calN$.
    \end{proof}

    Let $t$ and $y$ be as in Step~7.
    By Step~2 there exists $m \in M$ such that \[\mbox{
        $xm = x - \bfOne$ for all $x \in \openF$.
    }\]

    \begin{Step} 
        $(mt)^{2} = tm^{-1}$.
    \end{Step}
    \begin{proof}
        In cycle notation we have $t = (\infty \, \bfZero)(\bfOne \, -\bfOne)\ldots$
        so $mt = (\infty \, \bfZero \, \bfOne)\ldots$.
        Thus $(mt)^{3}$ fixes the three points $\infty, \bfZero$ and $\bfOne$.
        Hypothesis~(i) forces $(mt)^{3} = 1$.
        Since $t = t^{-1}$ the conclusion follows.
    \end{proof}

    \begin{Step} 
        There exists $z \in \calN$ with $z - \bfOne \in \calQ$.
    \end{Step}
    \begin{proof}
        By Lagrange's Theorem,
        $\calQ \cup \listset{\bfZero}$ is not a subgroup of $\openF\additive$
        so there exists $a,b \in \calQ \cup \listset{\bfZero}$
        with $a + b \in \calN$.
        Then $a, b \in \calQ$.
        Now $\calQ$ is a subgroup of $\openF\mult$ so \[
            \frac{a}{b} + \bfOne = \frac{a+b}{a} \in \calN \qtext{and} \frac{a}{b} \in \calQ.
        \]
        Put $z = a/b + \bfOne$.
    \end{proof}

    Let $z$ be as in Step~9.
    By Step~8, \[
        z (mt)^{2} = z (tm^{-1}).
    \]
    Using Step~7 we have
    \begin{align*}
        z (mt)^{2} &= (z - \bfOne)tmt = \biggl( -\frac{\bfOne}{z - \bfOne} \biggr)mt \\
                   &= \frac{-z}{z - \bfOne}t = \frac{z-1}{z} \text{\;\;or\;\;} \frac{-y(z - \bfOne)}{z} \\
                   &= \bfOne - \frac{\bfOne}{z} \text{\;\;or\;\;} -y + \frac{y}{z},\\
    \intertext{and}
        z (tm^{-1}) &= \frac{y}{z} + \bfOne.
    \end{align*}
    In both cases $y = -\bfOne$.
\end{proof}

Now $t \in D \setminus H$ so $t$ interchanges $\infty$ with $\bfZero$.
Step~7 implies that $t$ is the linear fractional transformation \[
    x \mapsto -\bfOne/x.
\]
In particular $t \in \psl{2}{\openF}$.
Recall that $M = G_{\infty}$.
By $2$-transitivity,
$G = \listgen{M,t}$.
For the same reason $\psl{2}{\openF} = \listgen{\psl{2}{\openF}_{\infty}, t}$.
Step~5(b) implies that $G = \psl{2}{\openF}$
and the proof is complete.

} 

%% file: bibliography.tex
\bibliographystyle{amsplain}

%
%